\documentclass[runningheads]{llncs}

\usepackage[T1]{fontenc}

\usepackage{amssymb}
\usepackage{amsfonts}
\usepackage{amsmath}
\usepackage{latexsym,epsfig}
\usepackage{graphicx} 
\usepackage[usenames,dvipsnames]{color} 
\usepackage{array} 
\usepackage{color} 
\usepackage{verbatim}
\usepackage[ruled,linesnumbered]{algorithm2e}  
\usepackage{url} 
\usepackage{enumerate} 
\usepackage{rotating}
\usepackage{hyperref} 
\usepackage{afterpage,lscape} 
\usepackage[misc,geometry]{ifsym} 
\usepackage[shortlabels]{enumitem} 
\usepackage{wasysym} 
\setlist[enumerate]{after={\bigskip}}
\setlist[itemize]{after={\bigskip}} 
\allowdisplaybreaks 
 
\newcommand{\EndProof}{\hspace{\stretch{1}} $\Box$}

\newcommand{\Z}{\mathbb{Z}}

\newcommand{\N}{\mathbb{N}}

\newcommand{\Semi}{\mathbb{S}}

\newcommand{\MREP}{\textsc{MinRep}} 
\newcommand{\MC}{\textsc{MinCost}} 
\newcommand{\GREP}{\textsc{GreedyRep}} 
\newcommand{\GRC}{\textsc{GreedyCost}}

\begin{document}

\title{Some Families of Greedy Numerical Semigroups} 

\titlerunning{Families of Greedy Numerical Semigroups}
%
\author{Arnau Messegu\'e-Buisan\inst{1}\orcidID{0000-0002-7425-7592} \\ 
Hebert P\'erez-Ros\'es\inst{1}\orcidID{0000-0002-3569-3885} } 

\authorrunning{Messegu\'e-Buisan and P\'erez-Ros\'es} 

\institute{Dept. of Mathematics \\ Universitat de Lleida, Lleida, Spain \\  
\email{\{arnau.messegue, hebert.perez\}@udl.cat}} 
\maketitle              
\begin{abstract}
The change-making problem was recently extended to sets of positive integers not containing the element $1$, and from there to numerical semigroups. A greedy numerical semigroup is defined as a numerical semigroup where the greedy representation of an element is optimal with respect to the number of summands. In this paper we identify some new families of greedy numerical semigroups. 

\end{abstract}
%

\section{Greedy sets and greedy numerical semigroups}
\label{sec:intro} 

The money-changing problem, or change-making problem, is a classic problem in combinatorial optimization. In its original formulation we are given a set of integer coin denominations $S = \{ s_1, s_2, \ldots, s_t \}$, such that $s_1 = 1 < s_2 < \ldots < s_t$, and a target amount $k$, and the goal is to represent $k$ using as few coins as possible.  

In formal terms we look for a \emph{payment vector} $\textbf{a} = (a_1, \ldots, a_t)$, such that 

\begin{align} 
	& \label{eq:cond1} \textstyle a_i \in \N_0 \ \mbox{ for all } i = 1, \ldots, t \\
	& \label{eq:cond2} \textstyle \sum_{i=1}^{t} a_i s_i = k, \\
	& \label{eq:cond3} \textstyle \sum_{i=1}^{t} a_i \mbox{ is minimal,}  
\end{align} 
where $\N_0$ denotes the set of nonnegative integers. 

The condition $s_1 = 1$ ensures that all positive integers are representable as a sum of the $s_i$. Under this original formulation, the change-making problem has been extensively studied (see, for instance \cite{Ada10,Cow08,KoZaks94,Pea05,PR24a,PR25-Istambul}, among others). A recent paper \cite{PR25} investigates a generalization of the problem to the case $s_1 \geq 1$ and $\gcd(s_1, \ldots, s_t)=1$. In that case almost all integers in $\N_0$ (except a finite number) are representable. More precisely, all the integers in $\langle s_1, \ldots, s_t \rangle$, the \emph{numerical semigroup} generated by $s_1, \ldots, s_t$, are representable. Formally, the numerical semigroup  $\Semi = \langle s_1, \ldots, s_t \rangle$ is the submonoid of $\N_0$ defined as
\begin{displaymath}
    \Semi = \langle s_1, \ldots, s_t \rangle = \{ n \in \N_0 : n = \alpha_1 s_1 + \cdots + \alpha_t s_t \mbox{ for some } \alpha_1, \ldots \alpha_t \in \N_0 \}, 
\end{displaymath} 
where $\gcd(s_1, \ldots, s_t)=1$. 

The complement $\N_0 \setminus \Semi$ is finite, and its largest element is the \emph{Frobenius number} of $\Semi$, denoted here as $F(\Semi)$. In $\Semi$ we can always find a \emph{minimal set of generators}, meaning that none of its proper subsets generates $\Semi$. The minimal set of generators is unique and finite, and its cardinality is the \emph{embedding dimension} of $\Semi$. For more concepts and results about numerical semigroups see \cite{Ra05} and \cite{RoGar09}, among other sources. 

The strategy that we use for representing a given $k \in \langle s_1, \ldots, s_t \rangle$ is the \emph{greedy strategy}, which proceeds by first choosing the largest possible generator $s_j$ that does not exceed $k$ and such that $k-s_j$ is representable (i.e. $k-s_j \in \Semi$), and then applying the same algorithm to $k-s_j$. Alternatively we can divide $k$ by $s_j$ and then proceed recursively with the remainder. This method is formally described in Algorithm \ref{alg:quasigreedypayment2}. 
\\\\ 
	\begin{algorithm}[H] 
		\SetKwInOut{Input}{Input}
		\SetKwInOut{Output}{Output} 
        \SetKwComment{Comment}{/* }{ */}
		\SetKw{DownTo}{downto}
		\vspace{.2cm}
		\Input{The set of denominations $S = \{ s_1, s_2, \ldots, s_t \}$, with $1 \leq s_1 < s_2 < \ldots < s_t$, $\gcd(s_1, s_2, \ldots, s_t) = 1$, and an element $k \in \langle S \rangle$.} 
		\Output{Greedy representation vector $\textbf{a} = (a_1, a_2, \ldots, a_t)$.} 
		\vspace{.2cm}

        \If{ $k \notin \langle S \rangle$ }{ \Return{ \lq\lq $k$ is not representable\rq\rq\; } } 
        Initialize $\textbf{a}$: $\textbf{a} \gets (0, 0, \ldots, 0)$\; 
        $i \gets t$\; 
        \While{$k > 0$} 
		{ 
            Let $q$ be the largest integer such that $k = q s_i + r$ and $r \in \langle S \rangle$\; 
			$a_i \gets q$\;
			$k \gets r$\;  
            $i \gets i-1$\; 
		} 
        \Return{ $\textbf{a}$ }\; 
		\caption{GREEDY REPRESENTATION METHOD} 
		\label{alg:quasigreedypayment2}  
	\end{algorithm}
\vspace{2mm} 

Notice that Algorithm \ref{alg:quasigreedypayment2} is also valid for $k=0$: it simply does not enter the while loop and returns the null  vector. If $k \in \langle S \rangle, \ k>0$, the algorithm always moves from $k$ to a smaller element $k' \in \langle S \rangle$, eventually arriving at $k=0$. Hence the loop will always terminate without having to introduce a termination condition on the variable $i$.  

Needless to say, the greedy strategy is not necessarily optimal. I.e. given $k \in \langle s_1, \ldots, s_t \rangle$, its greedy representation vector is not always minimal among all possible representation vectors of $k$. However, there do exist some sets $\{ s_1, \ldots, s_t \}$ such that the greedy representation vector of any $k \in \langle s_1, \ldots, s_t \rangle$ is indeed minimal. Before getting to that point let us introduce some definitions and notation: 
\begin{definition}
    \label{def:quasigreedycost} 
    Let $S=\{ s_1, s_2, \ldots, s_t \}$ be a set of generators (or \lq coin denominations\rq), with $1 \leq s_1 < s_2 \ldots < s_t$ and $\gcd(s_1, s_2, \ldots, s_t) = 1$, and let $k \in \langle S \rangle, \; k > 0$. The \emph{minimal representation} of $k$ with respect to $S$, denoted $\mbox{\MREP}_S(k)$ is a \lq payment\rq \  vector, or representation vector, $\textbf{a} = (a_1, a_2, \ldots, a_t)$ that satisfies Conditions (\ref{eq:cond1}), (\ref{eq:cond2}) and (\ref{eq:cond3}) above. If $\textbf{a}$ is a minimal representation of $k$, then  $\mbox{\MC}_S(k) = \sum_{i=1}^{t} a_i$. The \emph{greedy representation} of $k$ with respect to $S$, denoted $\mbox{\GREP}_S(k)$, is the representation vector $\textbf{a} = (a_1, a_2, \ldots, a_t)$ produced by Algorithm \ref{alg:quasigreedypayment2}, and $\mbox{\GRC}_S(k) = \sum_{i=1}^{t} a_i$, corresponding to the greedy representation vector.  
\end{definition} 

In the terminology of numerical semigroups $\mbox{\MREP}_S(k)$ is also called a \emph{minimum-length factorization of $k$}. For some considerations on factorizations and their lengths see \cite{Chap18}. 

\begin{definition}
    \label{def:quasigreedyset} 
    Let $S = \{ s_1, s_2, \ldots, s_t \}$ be a set of generators with $1 \leq s_1 < s_2 < \ldots < s_t$ and $\gcd(s_1, s_2, \ldots, s_t) = 1$, such that $\mbox{\GRC}_S(k) = \mbox{\MC}_S(k)$ for \emph{any} given $k \in \langle S \rangle$. Then $S$ will be  called \emph{ greedy}, and the semigroup $\Semi = \langle S \rangle$ will also be called \emph{ greedy}.  
\end{definition} 

The paper \cite{PR25} gives an algorithm to determine when a given set $S$ is greedy, and also some conditions for semigroups of embedding dimension $3$. These conditions will be briefly reviewed in Section \ref{sec:threegens}. 

Regarding the algorithmic identification of greedy sets a key concept is that of a \emph{counterexample}: 

\begin{definition}
    \label{def:counterexample} 
    Given a set of generators $S=\{ s_1, s_2, \ldots, s_t \}$, with $1 \leq s_1 < s_2 < \ldots < s_t$ and $\gcd(s_1, s_2, \ldots, s_t) = 1$, a \emph{counterexample} for $S$ is any representable integer $k > 0$, such that $\mbox{\GRC}_S(k) > \mbox{\MC}_S(k)$. 
\end{definition} 

Obviously, $S$ is greedy if and only if it does not contain any counterexamples. Therefore, in order to determine whether $S$ is (not) greedy, we must look for counterexamples within a certain interval. The following theorem formalizes this idea:

\begin{theorem}[\cite{PR25}]
    \label{theo:counterexample} 
    As in Definition \ref{def:counterexample} let $S=\{ s_1, \ldots, s_t \}$, with $t \geq 3$, be a set of generators such that $1 \leq s_1 < s_2 < \ldots < s_t$ and $\gcd(s_1, \ldots, s_t) = 1$. Then, $S$ is greedy if, and only if, $S$ does not have any counterexample $k$ in the interval 
    \begin{equation}
        \label{eq:criticalrange0}
        s_3 + s_1 + 2 \leq k \leq F(\Semi) + s_t + s_{t-1},   
    \end{equation}    
    where $F(\Semi)$ denotes the Frobenius number of $\Semi$. 
\end{theorem} 
\vspace{3mm} 

Note that there is no loss of generality by considering sets of cardinality $t \geq 3$ in Theorem \ref{theo:counterexample} above, given that sets of cardinality $1$ or $2$ are always greedy, and hence there are no counterexamples for $1 \leq t \leq 2$. 

The interval $\big[s_3 + s_1 + 2; \ F(\Semi) + s_t + s_{t-1} \big]$ given in Eq. (\ref{eq:criticalrange0}) is called the \emph{critical range}. Theorem \ref{theo:counterexample} suggests a procedure for checking if a given set $S$ is greedy, namely scan the critical range for a counterexample, and if no such counterexample can be found, then we may conclude that $S$ is greedy. Unfortunately, finding counterexamples involves determining the $\MC$, which is a hard problem. The solution to that dilemma involves the concept of a \emph{witness}: 

\begin{definition}
    \label{def:witness} 
    Given a set of generators $S=\{ s_1, s_2, \ldots, s_t \}$, with $1 \leq s_1 < s_2 < \ldots < s_t$ and $\gcd(s_1, s_2, \ldots, s_t) = 1$, a \emph{witness} for $S$ is any representable integer $k > 0$, such that $\mbox{\GRC}_S(k) > \mbox{\GRC}_S(k-s_i)+1$ for some generator $s_i < k$.  
\end{definition}

It turns out that the smallest counterexample is also a witness, and witnesses are easier to find. Now, with the aid of witnesses it becomes feasible to identify greedy sets: 
 
\begin{theorem}[\cite{PR25}] 
    \label{theo:witness}
    As in Definition \ref{def:witness} let $S=\{ s_1, \ldots, s_t \}$, with $t \geq 3$ be a set of generators such that $1 \leq s_1 < s_2 < \ldots < s_t$ and $\gcd(s_1, \ldots, s_t) = 1$. Then, $S$ is greedy if, and only if, $S$ does not have any witness $k$ in the interval 
    \begin{equation}
        \label{eq:criticalrange}
        s_3 + s_1 + 2 \leq k \leq F(\Semi) + s_t + s_{t-1},   
    \end{equation}    
    where $F(\Semi)$ denotes the Frobenius number of $\Semi$. 
\end{theorem} 
\vspace{3mm} 

Theorem \ref{theo:witness} finally provides a practical algorithm for checking if a given set of generators $S$ is greedy. The algorithm checks whether $\mbox{\GRC}_S(k) \leq \mbox{\GRC}_S(k-s_i)+1$ for all $k$ in the critical range, and all generators $s_i < k$. These ideas are formalized in Algorithm \ref{alg:search-witness}. 


	\begin{algorithm}[ht] 
		\SetKwInOut{Input}{Input}
		\SetKwInOut{Output}{Output}
		\SetKw{DownTo}{downto}
        \SetKw{To}{to} 
		\vspace{.2cm}
		\Input{The set of denominations $S = \{ s_1, s_2, \ldots, s_t \}$, with $t \geq 3$, $1 \leq s_1 < s_2 < \ldots < s_t$, $\gcd(s_1, s_2, \ldots, s_t) = 1$.}  
		\Output{TRUE if $\langle S \rangle$ is greedy, and  FALSE otherwise.} 
		\vspace{.2cm}

        $\Semi \gets \langle S \rangle$\;  
		\For{ $k \gets s_3 + s_1 + 2$ \To $F(\Semi) + s_t + s_{t-1}$} 
		{ 
			 $t' \gets$ Largest index $j$ such that $s_{j} \leq k$\;  
                \For{ $i \gets 1$ \To $t'-1$}  
                { 
                    \If{ $\mbox{\GRC}_S(k) > \mbox{\GRC}_S(k-s_i)+1$ }{ \Return{ FALSE }\; } 
                }
                
		} 
        \Return{ TRUE }\;
		\caption{DETERMINE WHETHER A SEMIGROUP DEFINED BY A GIVEN SET OF GENERATORS IS GREEDY}
		\label{alg:search-witness}  
	\end{algorithm}



Note that Algorithm \ref{alg:search-witness} implicitly requires the Frobenius number of $\Semi$ as an input. If we don't know the Frobenius number of $\Semi$, then it can be replaced by any upper bound (see, for instance, Section 2.3 of \cite{Ra05}, which mentions several such bounds).  

\section{Greedy semigroups of embedding dimension three} 
\label{sec:threegens} 

Now we will focus our attention on the special case of three generators. The paper \cite{PR25} gives a necessary condition:  

\begin{lemma}[\cite{PR25}] 
    \label{lemma:threegens} 
    Let $S = \{ s_1, s_2, s_3 \}$, with $1 < s_1 < s_2 < s_3$, and $\gcd(s_1, s_2, s_3) = 1$, so that $\Semi = \langle S \rangle$ is not greedy. Then, the smallest counterexample $k$ (and hence the smallest witness) has the form $k = s_2 y$, and it is a  solution of the Diophantine equation 
    \begin{equation}
        s_1 x + s_3 z = s_2 y, 
    \end{equation}
    where $x, y, z$ are positive integers, such that $y < x+z$. 
\end{lemma}  
\vspace{3mm}

With the aid of this lemma it was proved that sets of consecutive positive integers are greedy: 

\begin{theorem}[\cite{PR25}] 
    \label{theo:threeconsecutivegens} 
    Let  $\Semi = \langle n, n+1, n+2 \rangle$, with $n \geq 2$. Then, $\Semi$ is greedy. 
\end{theorem}
\vspace{3mm} 



We now want to expand the number of greedy families of cardinality three. We start with the following lemma:  

\begin{lemma}
\label{lemm:3weights}
Let $s_1,s_2,s_3$ and $a_1,a_2,a_3$ be positive integers with $s_1<s_2<s_3$. Suppose that $k=a_1s_1+a_2s_2+a_3s_3$, where $\displaystyle a_3 = \Bigl \lfloor \frac{k}{s_3} \Bigr \rfloor$, $\displaystyle a_2 = \Bigl \lfloor \frac{k-a_3s_3}{s_2} \Bigr \rfloor$ and $\displaystyle a_1 = \frac{k-a_2s_2-a_3s_3}{s_1}$. Then $$s_3 > s_1a_1+s_2a_2$$ and
$$s_2 > s_1a_1$$
\end{lemma}
\textit{Proof.}
Notice the following:
$$a_3 = \Bigl \lfloor \frac{k}{s_3} \Bigr \rfloor = \Bigl \lfloor \frac{a_3s_3+a_2s_2+a_1s_1}{s_3} \Bigr \rfloor = \Bigl \lfloor a_3+\frac{a_2s_2+a_1s_1}{s_3} \Bigr \rfloor$$
Then it must hold:
$$\frac{a_2s_2+a_1s_1}{s_3} < 1 \Longleftrightarrow a_2s_2+a_1s_1<s_3$$
Similarly:
$$a_2 = \Bigl \lfloor \frac{k-a_3s_3}{s_2} \Bigr \rfloor = \Bigl \lfloor \frac{a_2s_2+a_1s_1}{s_2} \Bigr \rfloor = \Bigl \lfloor a_2+\frac{a_1s_1}{s_2} \Bigr \rfloor$$
Then it must hold:
$$\frac{a_1s_1}{s_2} < 1 \Longleftrightarrow a_1s_1<s_2$$
\EndProof

\begin{proposition}
\label{prop:3weights}
Suppose that $\Delta_1,\Delta_2,\Delta_3$ are integers and $s_1,s_2,s_3, a_1,a_2,a_3$ are positive integers with $s_1<s_2<s_3$. Furthermore, suppose that: 
\begin{align}
&\Delta_1s_1+\Delta_2s_2+\Delta_3s_3 = 0 \\
& \Delta_j \geq -a_j \\
&s_3>s_1a_1+s_2a_2 \\
&s_2>s_1a_1\\
&\gcd(s_1,s_2)=1 
\end{align}
And,
\begin{align}s_3 \geq \frac{s_2^2}{s_1}
\end{align}
Then $\Delta_1+\Delta_2+\Delta_3 \geq 0$.
\end{proposition}
\textit{Proof.} 
From (7), (8) and (9) we have:
$$-\Delta_3 = \frac{\Delta_1s_1+\Delta_2s_2}{s_3} \geq -\frac{a_1s_1+a_2s_2}{s_3} > - \frac{s_3}{s_3} = -1$$
Therefore, $\Delta_3 \leq 0$. 
Now we distinguish two cases:
\begin{enumerate}
\item $\Delta_3 = 0$. In this case $\Delta_1s_1+\Delta_2s_2 = 0$ so that since $s_1,s_2$ are coprime by hypothesis then we must have $\Delta_1=s_2t$ and $\Delta_2=-s_1t$ with $t \in \mathbb{Z}$. Here:

\begin{enumerate}
\item If $t=0$ then $\Delta_1+\Delta_2+\Delta_3 = 0$

\item If $t > 0$ then $\Delta_1+\Delta_2+\Delta_3 \geq s_2-s_1 > 0$

\item If $t < 0$ then: 
$$-a_1 \leq \Delta_1 = s_2t \leq -s_2 < -s_1a_1 < a_1$$

\end{enumerate}

\item $\Delta_3 < 0$. Then $\Delta_1s_1+\Delta_2s_2 =-\Delta_3s_3\geq s_3$. From here one gets:
$$\Delta_2s_2 \geq s_3-\Delta_1s_1$$
Therefore:
\begin{align*}
\Delta_1+\Delta_2+\Delta_3 =& \Delta_1+\Delta_2- \frac{s_1\Delta_1+s_2\Delta_2}{s_3} = \\
=&\frac{\Delta_1(s_3-s_1)+\Delta_2(s_3-s_2)}{s_3} \\
\geq& (s_3-s_2)\left( \frac{s_3}{s_2}-\Delta_1\frac{s_1}{s_2}\right)+\Delta_1(s_3-s_1) \\
=& \frac{(s_3-s_2)s_3}{s_2}+\Delta_1\frac{s_3}{s_2}(-s_1+s_2)\\
=& \frac{s_3}{s_2}(s_3-s_2+\Delta_1(s_2-s_1))  \\
\geq& \frac{s_3}{s_2}((s_3-s_2)-a_1(s_2-s_1)) \\
\geq&\frac{s_3}{s_2}\left((s_3-s_2)-\frac{s_2}{s_1}(s_2-s_1)\right) \\
=& \frac{s_3}{s_2}\left(\frac{s_3s_1-s_2^2}{s_1}\right) \ \geq 0  
\end{align*}
\end{enumerate} 
\EndProof

\begin{theorem} 
\label{theo:three-gens1} 
If $S = \left\{s_1,s_2,s_3\right\}$ is a set of positive integers with $s_1<s_2<s_3$ such that $\gcd(s_1,s_2)=1$ and $\displaystyle s_3 \geq \frac{s_2^2}{s_1}$, then $S$ is greedy.
\end{theorem}

\textit{Proof.} Let $k=a_1s_1+a_2s_2+a_3s_3=A_1s_1+A_2s_2+A_3s_3$ with $a_3 = \lfloor k/s_3 \rfloor$, $a_2=\lfloor (k-a_3s_3)/s_2 \rfloor$ and $a_1 =  (k-a_3s_3-a_2s_2)/s_1$. Our claim is equivalent to showing that $A_1+A_2+A_3 \geq a_1+a_2+a_3$ in such a scenario. In order to see this let $\Delta_1 = A_1-a_1, \Delta_2 = A_2-a_2,\Delta_3 = A_3-a_3$. 

Then the following relationships are satisfied:
$$\Delta_1 s_1+ \Delta_2s_2+\Delta_3s_3 = (A_1-a_1)s_1+(A_2-a_2)s_2+(A_3-a_3)s_3= 0$$
$$\Delta_j = A_j-a_j \geq -a_j$$
Moreover by Lemma \ref{lemm:3weights}:
$$s_3>s_1a_1+s_2a_2$$
$$s_2>s_1a_1$$
Together with the conditions $\gcd(s_1,s_2)=1$ and $s_3 \geq s_2^2/s_1$ from the statement's hypothesis, we can apply Proposition \ref{prop:3weights} to deduce that: 
$$\Delta_1+\Delta_2+\Delta_3 \geq 0$$
Which is equivalent to what we wanted to see:
$$A_1+A_2+A_3\geq a_1+a_2+a_3$$

\EndProof  

An interesting special case in embedding dimension three is when the largest generator is the Frobenius number of the other two. Theorem \ref{theo:three-gens1} will help us to advance in that direction. We start with the following consequence of Theorem \ref{theo:three-gens1} above:  

\begin{corollary} 
    \label{coro:enlarged-with-Frobenius1}
    Let $S = \{ a, b, c \}$, with $4 \leq a < b < c$, $\gcd(a,b)=1$, $\displaystyle b \leq \frac{1}{2} \left( a^2-a + \sqrt{a^4-2a^3-3a^2} \right)$, and $c = F(a,b)$, i.e. $c = ab-a-b$. Then $\langle S \rangle$ is greedy. 
\end{corollary}

\textbf{Proof}: If $a \geq 4$ and $\displaystyle b \leq \frac{1}{2} \left( a^2-a + \sqrt{a^4-2a^3-3a^2} \right)$ then $\displaystyle c \geq \frac{b^2}{a}$, and by Theorem \ref{theo:three-gens1} above, $\langle S \rangle$ is greedy. 

\EndProof 
\\\\ 
We now want to remove the restriction on $b$. Ideally we would like to extend the previous result to all sets $S = \{ a, b, c \}$, with $1 < a < b < c$, $\gcd(a,b)=1$, and $c = F(a,b)$. Let's start with the following observation: 

\begin{lemma}
    \label{lemma:Frobenius-coprime}
    Let $a,b \in \N$, with $1 < a < b$ and $\gcd(a,b)=1$, and let $c = F(a,b) = ab-a-b$. Then $a$ and $c$ are coprime. 
\end{lemma} 

\textbf{Proof}: We want to check that 
\begin{displaymath}
    \gcd(a, c) = \gcd(a, ab - a - b) = 1
\end{displaymath} 
Now simplify  $c \ (\mbox{mod} \ a)$. Since $ab \equiv 0  \ (\mbox{mod} \ a)$ and $a \equiv 0  \ (\mbox{mod} \ a)$, we get: 
\begin{displaymath}
c = ab - a - b \equiv -b  \ (\mbox{mod} \ a), 
\end{displaymath} 
which means that 
\begin{displaymath}
    \gcd(a, c) = \gcd(a, -b) = \gcd(a, b) = 1, 
\end{displaymath} 
for the greatest common divisor is invariant under negation. 
\EndProof 

Now we proceed with our main goal:
\vspace{2mm} 
\begin{theorem}
    \label{theo:enlarged-with-Frobenius2} 
    Let $S = \{ a, b, c \}$, with $4 \leq a < b < c$, $\gcd(a,b)=1$,  and $c = F(a,b)$, i.e. $c = ab-a-b$. Then $\langle S \rangle$ is greedy. 
\end{theorem}

\textbf{Proof}: The case $\displaystyle b \leq \frac{1}{2} \left( a^2-a + \sqrt{a^4-2a^3-3a^2} \right)$ is already covered by Corollary \ref{coro:enlarged-with-Frobenius1}, hence we may assume that  $\displaystyle b > \frac{1}{2} \left( a^2-a + \sqrt{a^4-2a^3-3a^2} \right)$. According to Lemma \ref{lemma:threegens} we must look for multiples of $b$ which lie inside the critical range, and which can be represented as $ax+cz$, with $x, z \in \N$. In other words, we are looking for multiples $by \in \langle a,c \rangle$ (since by Lemma \ref{lemma:Frobenius-coprime}, $a$ and $c$ generate a numerical semigroup), such that $by=ax+cz$ with $x, z \in \N$, and such that $by$ also lies in the critical range of $\langle a,b,c \rangle$. Now, $c-1$ is an upper bound for $F(a,b,c)$, hence the critical range of $\langle a,b,c \rangle$ is a subinterval of $[a+c+2; \ c-1 + c + b] = [ab-b+2; \ 2ab-2a-b-1]$. So, we have 
\begin{align*}
    by &= ax + cz = ax + z(ab - a - b) \\
        &= ax + abz - az - bz = a(x - z) + bz(a - 1), 
\end{align*} 
whence
\begin{displaymath}
    y = \frac{a(x - z)}{b} + z(a - 1). 
\end{displaymath} 
Since $y$ is an integer, $b$ must divide $a(x-z)$. Since $\gcd(a,b)=1$, then $b$ must divide $(x-z)$. Thus, $x-z =bt$ and $x = z + bt$ for some $t \in \Z$. We must also ensure that $x > 0$ and $z > 0$, so $z + bt > 0$ implies that $t > -z/b$. 

So, 
\begin{displaymath}
     by = a(x - z) + bz(a - 1) = abt + bz(a - 1), 
\end{displaymath}
hence  
\begin{displaymath}
    y = at + z(a - 1). 
\end{displaymath} 
Now we introduce the upper bound provided by the critical range 
\begin{displaymath}
    by = b(at + z(a - 1)) \leq 2ab-2a-b-1, 
\end{displaymath}
which means that 
\begin{displaymath}
    at + z(a - 1) \leq 2a - 1 - \frac{2a + 1}{b}. 
\end{displaymath} 
Taking into account that $\displaystyle b > \frac{1}{2} \left( a^2-a + \sqrt{a^4-2a^3-3a^2} \right)$ we have that $2a+1 < b$ for $a \geq 4$. Hence
\begin{displaymath}
    at + z(a - 1) \leq 2a - 2. 
\end{displaymath} 
We now try values of $t \in \Z$ such that both $x = z + bt > 0$ and $z > 0$, and the upper bound above is satisfied. Let's start by checking negative values of $t$. 

Let's plug in $t=-1$, for instance, which makes $x = z - b$. On one hand, to ensure that $x > 0$ we must have $z > b$. On the other hand,  $-a + z(a - 1) \leq 2a - 2$ implies that $\displaystyle z \leq \frac{3a - 2}{a - 1}$. That gives us a maximum value of $z = 2$, which is a contradiction. Thus, $t=-1$ is impossible. 

More generally, let $t$ be a negative integer, and set $j=-t$. Thus, $j$ is a positive integer. Again we must have 
\begin{align*}
     at + z(a - 1) &\leq 2a - 2 \\
     z(a - 1) - aj &\leq 2a - 2
\end{align*} 
and $z>bj$. Therefore, 
\begin{displaymath}
    bj < z \leq j+2 + \frac{j}{a-1} 
\end{displaymath} 
The set of inequalities $\displaystyle bj < j+2 + \frac{j}{a-1}$ and $a<b$ only has a potential feasible solution, namely 
\begin{displaymath}
    a \geq 2, \; j < \frac{2a-2}{ab-a-b}
\end{displaymath} 
However, for $a \geq 4$ we have that $\displaystyle \frac{2a-2}{ab-a-b} < 1$, which is a contradiction with our assumptions about $j$. Therefore, $t$ cannot be a negative integer. 

Let us now check the case of $t=0$. We now have $z(a-1) \leq 2a-2$, which implies that $\displaystyle z \leq \frac{2a-2}{a-1} = 2$. Hence there are two possible values for $z$, namely $z=1$ or $z=2$. If $z=1$ we get $y=a-1$, but $by=b(a-1) < ab-b+2$ (i.e., $b(a-1)$ falls outside the critical range). On the other hand, if $z=2$ we get $y=2(a-1)$, and $by=2b(a-1)$ is inside the critical range, so it is a valid candidate for being a witness, and hence a counterexample. 

Finally, let's check positive values of $t$, starting with $t=1$. In that case we have $a+z(a-1) \leq 2a-2$, which implies that $z<1$, which is impossible. If $t \geq 2$ then $\displaystyle z \leq 2 - \frac{a}{a-1}t$. The term $\displaystyle \frac{a}{a-1}t > 2$ for $a \geq 4$ and $t \geq 2$, which again means that $z<1$, which contradicts our assumptions. 

In summary, there is only one feasible multiple $by$ that falls inside the critical range, namely $by = 2b(a-1) = 2a+2c$, which might be a potential counterexample. However, 
\begin{displaymath}
    \mbox{\GRC}_S\bigl( 2b(a-1) \bigr)=\mbox{\MC}_S\bigl( 2b(a-1) \bigr) = 4.
\end{displaymath}
Hence, $2b(a-1)$ is not a counterexample. This proves that $\langle S \rangle$ is greedy, as claimed. 

\EndProof 

\section{Semigroup families generated by the Fibonacci sequence} 
\label{sec:totallygreedy} 

If we have a greedy set $S$, a straightforward question is to determine which subsets of $S$ are also greedy. We now introduce the following definition: 

\begin{definition}
	\label{def:totallygreedyset} 
	Let $S = \{ s_1, s_2, \ldots, s_t \}$ be a set such that $1 \leq s_1 < s_2 \cdots < s_t$. We say that $S$ is \emph{totally greedy}\footnote{Also called normal, or totally orderly.} if every prefix subset $\{ s_1, s_2, \ldots, s_i \}$, where $1 \leq i \leq t$, is greedy. 
\end{definition} 

Definition \ref{def:totallygreedyset} can be extended to infinite sequences in a straightforward way:   

\begin{definition}
	\label{def:totallygreedysequence} 
	Let $S = \{ s_n \}_{n=1}^{\infty}$ be an integer sequence, with $s_1 \geq 1$ and $s_i < s_{i+1}$ for all $i \in \N$. We say that $S$ is \emph{totally greedy} (or simply, greedy) if every prefix subset $S^{(t)} = \{ s_1, s_2, \ldots, s_t \}$ is greedy, for any arbitrary $t \in \N$.  
\end{definition} 

Totally greedy sequences having $s_1 = 1$ were investigated in \cite{PR24a,PR25-Istambul}. All the results contained in \cite{PR24a,PR25-Istambul} have been condensed in a single paper in \cite{PR24b}. Now the aim is to expand these investigations to more general sequences, where $s_1 \geq 1$. One such sequence is the (shifted) Fibonacci sequence.

 Let $\big\{ F_n \big\}_{n=0}^\infty = \{ 1, 1, 2, 3, 5, 8, 13, \ldots \}$ denote the Fibonacci sequence, and define $\big\{ _rF_n \big\}_{n=0}^\infty = \{ F_r, F_{r+1}, F_{r+2}, \ldots \} = \big\{ F_{r+n} \big\}_{n=0}^\infty$, where $r \in \N_0$, i.e. $\big\{ _rF_n \big\}_{n=0}^\infty$ is the Fibonacci sequence shifted $r$ places to the left. Thus, $\big\{ _0F_n \big\}_{n=0}^\infty$ is just the original Fibonacci sequence $\big\{ F_n \big\}_{n=0}^\infty$. 

In \cite{PR24a} it was proved that $\big\{ _1F_n \big\}_{n=0}^\infty$ is totally greedy. Actually, the greediness of the (full) Fibonacci sequence has been known for a long time under a different guise. The greedy representation of a number with respect to $\big\{ _1F_n \big\}_{n=0}^\infty$ (which also turns out to be optimal in the sense that the number of digits is minimal) is known as the Zeckendorf representation, after the Belgian amateur mathematician Edouard Zeckendorf, who (re)discovered it in 1972 \cite{Zeck72}. The Zeckendorf representation of a number $k$ has two significant properties:
\begin{enumerate}
    \item It does not contain two consecutive Fibonacci numbers, $F_j$ and $F_{j+1}$, and 
    \item It does not contain two occurrences of a Fibonacci number $F_j$.
\end{enumerate} 
Now we want to investigate $\big\{ _rF_n \big\}_{n=0}^\infty$, for $r>1$. In particular we are going to prove the following 

\begin{theorem}
    \label{theo:Fibonacci2}
    The sequence $\big\{ _2F_n \big\}_{n=0}^\infty = \{ 2, 3, 5, 8, 13, \ldots \}$ is totally greedy. 
\end{theorem}

Before setting out to prove Theorem \ref{theo:Fibonacci2} let us take note of a few facts. Let $S = S^{(t)} = \{ 2,3,5,\ldots, F_t \}$ be a prefix subset of the shifted Fibonacci sequence $\big\{ _2F_n \big\}_{n=0}^\infty$, where $t > 3$ is fixed. Note that $\gcd(S) = 1$, so that $S$ generates a numerical semigroup $\Semi$ with Frobenius number $F(\Semi) = 1$. The critical range of $\Semi$ is 
\begin{displaymath}
    \big[ 9; \ F_{t} + F_{t-1} + 1 \big] = \big[ 9; \ F_{t+1} + 1 \big]
\end{displaymath} 


In order to prove Theorem \ref{theo:Fibonacci2} we will need some preliminary results, starting with the following lemma (see \cite{PR25}): 

\begin{lemma}[\cite{PR25}]
\label{lemma:kozaks}
    Let $\Semi = \langle S \rangle$ be a numerical semigroup generated by the set $S = \{ s_1, \ldots, s_t \}$. Additionally, let $k \in \Semi$ and $s_i$ be such that $k-s_i \in \Semi$. Then 
    \begin{equation}
        \label{eq:lemmakozen}
        \mbox{\MC}_S(k) \leq \mbox{\MC}_S(k-s_i)+1, 
    \end{equation}
    with equality holding if, and only if, there exists an optimal representation of $k$ that uses the generator $s_i$. 
\end{lemma}
\vspace{3mm}  

The greedy representation of a number with respect to $\big\{ _2F_n \big\}_{n=0}^\infty$ is similar to the Zeckendorf representation (i.e. the greedy representation with respect to $\big\{ _1F_n \big\}_{n=0}^\infty$). The following definition and the ensuing lemmas look into the properties of our Zeckendorf-type representation. 

\begin{definition} 
\label{def:properties}
    Given a positive integer $k$, and a representation $r(k)$ in terms of Fibonacci numbers. We say that $r(k)$ has  
    \begin{itemize}
        \item \textbf{Property 1}: If $r(k)$ does not contain two consecutive Fibonacci numbers, $F_j$ and $F_{j+1}$, and 
        \item \textbf{Property 2}: If $r(k)$ does not contain two occurrences of a Fibonacci number $F_j$, except possibly $F_2 = 2$ or $F_3 = 3$, which may appear twice (but not at the same time). 
    \end{itemize} 
\end{definition} 

It turns out that the greedy representation has both properties:
\begin{lemma}
    \label{lemma:properties} Let $k \geq 2$ be an arbitrary positive  integer. Then, the greedy representation of $k$ with respect to $\big\{ _2F_n \big\}_{n=0}^\infty$ satisfies Properties 1 and 2 of Definition \ref{def:properties} above. 
\end{lemma}

\textbf{Proof}: We will proceed by induction. The statement clearly holds for $k = 2, \ldots, 14$, which accounts for the induction base. Now let $S^{(t)} = \{ 2,3,5,\ldots, F_t \}$, where $t \geq 6$, and $\displaystyle k \in [F_t + 2; F_{t+1} + 1]$. Our induction argument is based on the parameter $t$.  Note that $t \geq 6$ implies $k > 14$. 

 As the induction hypothesis we assume that the statement is valid for all $2 \leq j < t$. Let us now prove that it also holds for $t$. This induction scheme will be used several times throughout this section. It is similar to the scheme employed in other papers dealing with Zeckendorf-type representations, such as \cite{Hog72}, but we have attempted to make it more precise. 
 
 Since $\displaystyle k \in [F_t + 2; F_{t+1} + 1]$, the greedy  representation of $k$ involves $F_t$, with the exception of $k = F_{t+1}$, whose greedy representation is just $F_{t+1}$ itself, and this representation obviously satisfies Properties 1 and 2. Note that in the case $k = F_{t+1} + 1$, no representation of $k$ can contain $F_{t+1}$, so it must contain $F_t$. 
 


That leaves us with two cases: $k = F_{t+1}+1$ or $\displaystyle k \in [F_t + 2; F_{t+1} - 1]$. In the first case the largest \lq digit\rq \ of the greedy representation of $k$ is $F_t$, and the remainder is $F_{t+1}+1 -F_t = F_{t-1}+1$, so that the second largest digit will be $F_{t-2}$, which agrees with Properties 1 and 2. 

Now we may apply the induction hypothesis on $F_{t-1}+1$. This induction process may end up either in $6=3+3$ or $4=2+2$, which are the two exceptions foreseen by Property 2. 

In the second case $\displaystyle k \in [F_t + 2; F_{t+1} - 1]$. Again, the largest \lq digit\rq \ of the greedy representation of $k$ is $F_t$, and the remaining digits are determined from $k-F_t$. Now, $k - F_t \leq F_{t-1} - 1$, hence the second largest digit will be $\leq F_{t-2}$, in accordance with Properties 1 and 2. 

Now we may apply the induction hypothesis on $k-F_t$ to complete the proof. In this case the induction process may also end up in $6=3+3$ or $4=2+2$, but not necessarily so. 


\EndProof 


\begin{lemma}
    \label{lemma:colex} Let $k \geq 2$ be an integer. Then, the greedy representation of $k$ with respect to $\big\{ _2F_n \big\}_{n=0}^\infty$ is the largest representation of $k$ in  the colexicographic order. 
\end{lemma}

\textbf{Proof}: If we reverse the order of the digits, then the greedy representation is the largest representation with respect to the lexicographic order (see \cite{Pea05}).  

\EndProof 

Another result that has to do with the greedy representation is the following:
\begin{lemma}
    \label{lemma:sum1} Suppose that we have an integer $k$ represented in $\big\{ _2F_n \big\}_{n=0}^\infty$, such that the representation satisfies Properties 1 and 2 of Definition  \ref{def:properties}. Suppose also that the largest Fibonacci number that appears in the representation of $k$ is $F_t$. Then $k \leq F_{t+1} + 1$. 
\end{lemma}

\textbf{Proof}: We proceed by induction on $t$. As the base case we may verify the assertion for all $t \leq 6$. Now let $t > 6$, and suppose that the assertion holds for all $j < t$. Let $F_t$ be the largest Fibonacci number that appears in the representation of $k$, and suppose that the representation satisfies Properties 1 and 2 of Definition \ref{def:properties}. By Property 2, $F_t$ appears only once. By Property 1, the second largest Fibonacci number appearing in the representation of $k$ is $F_{t-2}$, or a smaller one. Then, by the induction hypothesis, $k - F_t \leq F_{t-1} + 1$, and the result follows. 



\EndProof 

Now we discuss the uniqueness of the representation:

\begin{lemma}
    \label{lemma:uniqueness} Let $r(k)$ be a representation of $k$ in $\big\{ _2F_n \big\}_{n=0}^\infty$, which satisfies Properties 1 and 2 of Definition \ref{def:properties}. Then $r(k) = \GREP_S(k)$. 
\end{lemma}

\textbf{Proof}: Again we will proceed by induction. The statement can be easily verified by hand for $k = 2, \ldots, 6$. That accounts for the induction base. Now let $t \geq 4$, $S^{(t)} = \{ 2,3,5,\ldots, F_t \}$ and $\displaystyle k \in [F_t + 2; F_{t+1} + 1]$. 

We use induction on $t$. As the induction hypothesis we assume that the statement is valid for all $2 \leq j < t$. Let us now prove that it also holds for $t$. Since $\displaystyle k \in [F_t + 2; F_{t+1} + 1]$, there exists a representation of $k$ that involves $F_t$. In fact, by Lemma \ref{lemma:sum1} any representation $r(k)$ that satisfies Properties 1 and 2 \emph{must} involve $F_t$, except for $k = F_{t+1}$. Indeed, if $F_{t-1}$ is the largest \lq digit\rq \ in $r(k)$, then $k \leq F_t + 1$. 

In the case $k = F_{t+1}$, this number can be represented by itself, and this representation obviously satisfies Properties 1 and 2. Note that in the case $k = F_{t+1} + 1$, no representation of $k$ can contain $F_{t+1}$, so it must contain $F_t$. 

Thus, let $\displaystyle k \in [F_t + 2; F_{t+1} + 1]$, $k \neq F_{t+1}$. Then $r(k)$ contains $F_t$, so $r(k)$ agrees with $\GREP_S(k)$ in the most significant digit. The remaining digits are determined from $k-F_t$. Now, $k - F_t \leq F_{t-1} + 1$, hence we may apply the induction hypothesis with $j = t-2$, and the result follows. 

\EndProof 

In other words, $\GREP_S(k)$ is the unique representation of $k$ in terms of $\big\{ _2F_n \big\}_{n=0}^\infty$ that satisfies Properties 1 and 2. Now we show that $\GREP_S(k)$ is optimal: 

\begin{lemma} 
    \label{lemma:summands} Let $k \geq 2$ be an integer, which is to be represented in $\big\{ _2F_n \big\}_{n=0}^\infty$. Then, $\GREP_S(k)$ minimizes the number of summands. 
\end{lemma}

\textbf{Proof}: Let $r(k)$ be any representation of $k$ in $\big\{ _2F_n \big\}_{n=0}^\infty$ (not necessarily satisfying Properties 1 and 2). We may repeatedly apply the following \emph{rewriting} or \emph{reduction rules} on $r(k)$: 
\begin{enumerate} 
    \item \textbf{Eliminate consecutive Fibonacci summands}: If $r(k)$ contains a pair of consecutive Fibonacci summands, $F_j$ and $F_{j+1}$, for any $j \geq 2$, then replace the pair by $F_{j+2}$. This operation reduces the number of summands by one. 
    \item \textbf{Eliminate repetitions of a Fibonacci number}: If $r(k)$ contains two (or more) occurrences of a Fibonacci number $F_j$, for any $j \geq 4$, then replace $2F_j$ by $F_{j+1} + F_{j-2}$. This operation does not reduce the number of summands, but it creates opportunities for applying the first rule.  
    \item \textbf{Eliminate trio of $2$'s}: If $r(k)$ contains three (or more) occurrences of $2$, then replace a trio of $2$'s by two occurrences of $3$. 
    \item \textbf{Eliminate trio of $3$'s}: If $r(k)$ contains three (or more) occurrences of $3$, then replace a trio of $3$'s by the trio $5+2+2$. 
\end{enumerate}

Note that the application of any of these four rules yields a larger representation of $k$ in colexicographic order, and since the number of representations of $k$ is finite, the procedure terminates (i.e. it is an algorithm). When the procedure terminates the final representation $r_0(k)$ so obtained satisfies Properties 1 and 2 above, and it has at most the same number of summands as the initial representation $r(k)$. By Lemma \ref{lemma:uniqueness}, $r_0(k) = \GREP_S(k)$. 

\EndProof 
\vspace{2mm} 

So far we have investigated representations of integers in terms of the full shifted sequence $\big\{ _2F_n \big\}_{n=0}^\infty$, and we can see that the greedy representation in terms of this sequence is optimal in the number of summands. This result is similar to Zeckendorf's Theorem, and our proof follows the traditional approaches for proving Zeckendorf's Theorem. Apparently, Theorem \ref{theo:Fibonacci2} is stronger than that: It states that \emph{any} prefix subset of $\big\{ _2F_n \big\}_{n=0}^\infty$ is greedy. However, we can show that the previous results actually imply Theorem \ref{theo:Fibonacci2}. 
\\\\ 
\textbf{Proof of Theorem \ref{theo:Fibonacci2}}: Again, let $S^{(t)} = \{ 2,3,5,\ldots, F_t \}$ be a prefix subset of the shifted Fibonacci sequence $\big\{ _2F_n \big\}_{n=0}^\infty$, with $t \geq 2$, and let $\Semi^{(t)} = \langle S^{(t)} \rangle$ be the numerical semigroup generated by $S^{(t)}$. We are going to show that $\Semi^{(t)}$ is greedy for all $t$ by induction on $t$. It is clear that $\Semi^{(2)}$ and $\Semi^{(3)}$ are greedy (see \cite{PR25}), and we may easily verify that $\Semi^{(4)}$ is also greedy with the aid of Theorem \ref{theo:three-gens1}. That establishes the induction base. 

Now let $t \geq 5$, and let us assume that $\Semi^{(j)}$ is greedy for all $2 \leq j < t$. In order to show that $\Semi^{(t)}$ is also greedy we have to show that $\mbox{\GRC}_S(k) = \mbox{\MC}_S(k)$ for all $k$ in the critical range $\displaystyle \big[ 9; \ F_{t+1} + 1 \big]$. However, we may restrict ourselves to those integers $k$ that have a representation involving $F_t$ (excluding $F_t$ itself, whose representation is trivial), since smaller values of $k$ are dealt with by the induction hypothesis. Hence we may restrict ourselves to integers $\displaystyle k \in \big[ F_t + 2; \ F_{t+1} + 1 \big]$. Now, for integers in the range $\displaystyle \big[ F_t + 2; \ F_{t+1} + 1 \big]$ we may apply Lemma \ref{lemma:summands}, and the result follows. 

\EndProof  
\vspace{2mm}  

Now that we know that $S^{(t)}$ is greedy for any $t \geq 2$, let us examine the form of $\GREP_S(k)$ with respect to $S^{(t)}$. Its properties are quite similar to Properties 1 and 2, with some minor changes: 

\begin{itemize}
    \item \textbf{Property 1b}: It does not contain two consecutive Fibonacci numbers, $F_j$ and $F_{j+1}$, with the possible exception of $F_{t-1}$ and $F_t$, and 
    \item \textbf{Property 2b}: It does not contain two occurrences of a Fibonacci number $F_j$, with the possible exceptions of $F_2 = 2$ and $F_3 = 3$, which may appear twice (but not at the same time), and $F_t$, which may appear an arbitrary number of times. 
\end{itemize} 
Similarly, the rewriting procedure in the proof of Lemma \ref{lemma:summands} can also be adapted to the case of $S^{(t)}$. We start with \emph{any} representation $r(k)$ of $k$ in terms of $S^{(t)}$, and then apply the (modified) set of rules: 

\begin{enumerate} 
    \item \textbf{Eliminate consecutive Fibonacci summands}: If $r(k)$ contains a pair of consecutive Fibonacci summands, $F_j$ and $F_{j+1}$, for any $2 \leq j \leq t-2$, then replace the pair by $F_{j+2}$. 
    \item \textbf{Eliminate repetitions of a Fibonacci number}: If $r(k)$ contains two (or more) occurrences of a Fibonacci number $F_j$, for any $4 \leq j \leq t-1$, then replace $2F_j$ by $F_{j+1} + F_{j-2}$. 
    \item \textbf{Eliminate trio of $2$'s}: If $r(k)$ contains three (or more) occurrences of $2$, then replace a trio of $2$'s by two occurrences of $3$. 
    \item \textbf{Eliminate trio of $3$'s}: If $r(k)$ contains three (or more) occurrences of $3$, then replace a trio of $3$'s by the trio $5+2+2$. 
\end{enumerate}



Unfortunately, Theorem \ref{theo:Fibonacci2} cannot be generalized in a straightforward manner to arbitrary shifts $r$. Take for instance $r=3$, and take the truncated set $S = \{ 3, 5, 8 \}$. It is easy to verify that $k=20$ is a counterexample for $S$, since 
\begin{align*}
    \GRC_S(20) &= 5, \; \GREP_S(20) = 8+3+3+3+3, \; \mbox{ whereas} \\
    \MC_S(20) &= 4, \; \MREP_S(20) = 5+5+5+5. 
\end{align*}

Consequently, $S$ is not greedy, and the shifted sequence 
\begin{displaymath}
    \big\{ _3F_n \big\}_{n=0}^\infty = \{ 3, 5, 8, 13, 21, \ldots \} 
\end{displaymath}
is \emph{not} totally greedy. 


\section{Some open problems}
\label{sec:open}

In this section we collect several interesting open questions that have arisen throughout our research. 

In Section \ref{sec:threegens} we expand our catalog of greedy numerical semigroups with embedding dimension three, which was started in \cite{PR25}. In particular, Theorem \ref{theo:enlarged-with-Frobenius2} allows us to augment a numerical semigroup of embedding dimension two into a greedy one with embedding dimension three, by adding the Frobenius number. 

Computational experiments suggest that Theorem \ref{theo:enlarged-with-Frobenius2} can be extended to all $a \geq 2$. However, the assumption $a \geq 4$ plays a significant role in the proof. It would be interesting to find another proof of Theorem \ref{theo:enlarged-with-Frobenius2} that doesn't make use of this assumption, and thus extend the theorem to all $a \geq 2$, if possible. 

We are also interested in the generalization of Theorem \ref{theo:enlarged-with-Frobenius2} to higher embedding dimensions. Is it always true that if $\Semi = \langle S \rangle$ is greedy, then $\langle S \cup F(\Semi) \rangle$ is also greedy?

In Section \ref{sec:totallygreedy} we have proved that the Fibonacci number sequence shifted two places to the left is totally greedy (Theorem \ref{theo:Fibonacci2}). The proof follows an ad hoc inductive argument that involves several tiny details and special cases. Generalizing this argument to other sequences or families of sequences might be tedious. Here we need more powerful tools, such as the so-called \emph{One-point Theorem}, which applies to sequences starting with $1$ (see \cite{Ada10,PR24a,PR25-Istambul}). 

Additionally, we saw that if the same Fibonacci sequence is shifted \emph{three} places to the left, then it loses the greedy property. A natural question is: Given a totally greedy sequence having the first term equal to $1$, which shifts preserve the greedy property?  

\section*{Acknowledgements}

This work has been partially supported by the Spanish Ministry of Science and Innovation MCIN$/$AEI 10.13039/501100011033 (contract MATSE PID2024-156636NB-C22). 



\end{document}